 	\documentclass[10pt,a4paper]{amsart}
 	\usepackage{amscd,amssymb,amsopn,amsmath,amsthm,graphics,amsfonts,enumerate,verbatim,calc}
 	\usepackage[dvips]{graphicx}

 	\usepackage{mathpazo}
 	\usepackage{color}
 	\usepackage{latexsym}
 	\usepackage{amsthm,amsfonts,amssymb,mathrsfs}
 	\usepackage{rotating}
 	\usepackage[leqno]{amsmath}
 	\usepackage{xspace}
 	\usepackage[all]{xy}
 	\usepackage{longtable}
 	\headsep=5 mm \numberwithin{equation}{section}
 	\newtheorem{theorem}{Theorem}[section]
 	\newtheorem{lemma}[theorem]{Lemma}
 	
 	\newtheorem{proposition}[theorem]{Proposition}
 	\newtheorem{corollary}[theorem]{Corollary}

 	\theoremstyle{definition}
 	\newtheorem{definition}[theorem]{Definition}
 	\theoremstyle{remark}
 	\newtheorem{remark}[theorem]{Remark}
 	\newtheorem{fact}[theorem]{Fact}
 	\newtheorem{example}[theorem]{Example}
 	\newtheorem{observation}[theorem]{Observation}
 	
 	\newtheorem{question}[theorem]{Question}
 	\newtheorem{conjecture}[theorem]{Conjecture}

 	\newcommand{\codim}{\operatorname{codim}}
 	\newcommand{\Ass}{\operatorname{Ass}}

 	\newcommand{\im}{\operatorname{im}}

 	\newcommand{\gr}{\operatorname{grade}}

 	\newcommand{\Spec}{\operatorname{Spec}}

 	\newcommand{\rad}{\operatorname{rad}}

 	\newcommand{\Ht}{\operatorname{ht}}

 	\newcommand{\Sing}{\operatorname{Sing}}
 	
 	\newcommand{\HH}{\operatorname{H}}

 	\newcommand{\V}{\operatorname{V}}

 	\newcommand{\Ext}{\operatorname{Ext}}

 	\newcommand{\Supp}{\operatorname{Supp}}
 	
 	\newcommand{\Tor}{\operatorname{Tor}}
 	
 	\newcommand{\Hom}{\operatorname{Hom}}
 	
 	\newcommand{\Ann}{\operatorname{Ann}}

 	\newcommand{\depth}{\operatorname{depth}}

 	\newcommand{\Min}{\operatorname{Min}}

 	\newcommand{\CM}{CM }

 	\newcommand{\lo}{\longrightarrow}
 	\newcommand{\fm}{\mathfrak{m}}
 	\newcommand{\fp}{\mathfrak{p}}
 	\newcommand{\fq}{\mathfrak{q}}

\begin{document}
 		
 		\author[]{Mohsen Asgharzadeh}
 		
 		\address{}
 		\email{mohsenasgharzadeh@gmail.com}

 		\title[ ]
 		{weak normality, Gorenstein and Serre's conditions }

 		
 		\begin{abstract}
 		We compute the associated prime ideals of the normalization modulo the ring, and 	
 establish connections between  different types of generalizations		
 	 (resp. specializations) of the normalization.
 		This has some applications. For example, we present a characterization of quasi-normality, in the format that was conjectured by Vasconcelos. Also, we show in some cases, the normalization modulo the ring,
 		is homologically reduced. This provides a partial answer to a conjecture of Matlis.
 		\end{abstract}
 		
\maketitle

 		\section{Introduction}
 	
In this note $R$ is a commutative noetherian ring  with normalization $\overline{R}$.
Let $X$ be any subset of $\Spec (R)$. By $X^{\geq i}$ (resp. $X^i$) we mean $\{x\in X:\Ht(x)\geq i\}$ (resp.  $\{x\in X:\Ht(x)= i\}$).  By $\Ass(-)$ we mean the set of all associated prime ideals  of the $R$-module $(-)$.

  It is not hard to find a $1$-dimensional ring such that  $\Ass(\frac{\overline{R}}{R})$ consists of infinitely many height one prime ideals, see \cite{KOL}.	Kollar posted a question on the existence of integral domains of every dimension such that
  $\Spec(R)^1\subset\Ass(\frac{\overline{R}}{R})$, see  \cite[Question 54]{KOL}. We remark that there is a $2$-dimensional  quasi-normal local domain   such that
  $\Spec(R)^1= \Ass(\frac{\overline{R}}{R})$. Recall from \cite[Question 56]{KOL} that
  
\begin{question}  
 	 What can one say about  $\Ass(\frac{\overline{R}}{R})^{\geq2}$?
  \end{question}

Concerning Question 1.1, first recall that there are situations (integral domain or not) for which $\Ass(\frac{\overline{R}}{R})^{\geq 2}$
is infinite.   Then, for a (reduced) ring $R$, we show:
  	\begin{observation}
  	  $R$ is $(S_2)$ iff $\Ass(\frac{\overline{R}}{R})^{\geq 2}=\emptyset$ iff  $\{\fp\in\Spec(R):\fp^{\ast\ast}=\fp\}^{\geq 2}=\emptyset$.
  		\end{observation}
This  implies that  $\Ass^{\geq 2}(\frac{\overline{R}}{R})$ is a good candidate for measuring Serre's $(S_2)$
condition, even the ring is not reduced.
 In fact, $\Ass^{\geq 2}(\frac{\overline{R}}{R})$ provides a new method to construct  the  $(S_2)$ locus, see Corollary \ref{gr}.
 We compute $\Ass(\frac{\overline{R}}{R})^{\geq 2}$ in  almost and generalized
 Cohen--Macaulay cases, see Proposition \ref{acm} and \ref{ge}. Additional examples are presented in \ref{inters} and \ref{remint}.
 Over a quasi-normal ring, Vasconcelos \cite[Theorem 2.4]{wol} proved that an ideal (of positive grade)  is reflexive
 if and only if all of its associated primes are of height one. He noted
 that possibly this is characteristic of quasi-normality (see \cite[Page 271]{wol}). 
As an application, we confirm  Vasconcelos' prediction:

 \begin{theorem}
	Let $R$ be  an integral domain and suppose	a nonzero ideal  is reflexive
	if and only if all of its associated primes are of height one. Then  $R$ is  quasi-normal.
\end{theorem}
Recall that the $(S_2)$ condition has an essential role in \cite{gt}.
Gerco and Traverso \cite{gt} gave a seminormal  ring which is not $(S_2)$, and they implicitly asked when is a  seminormal ring 
 $(S_2)$? As weakly normal rings are seminormal, and as another consequence of Observation 1.2, we observe:

\begin{corollary}\label{1.4}
	Let 	$(R, \fm,k)$  be   complete, weakly normal domain and of prime characteristic with $k=\overline{k}$.
	Then $\depth(R)\geq\min\{2,\dim R\}$. Also, $R$ is   $(S_2)$
	provided $R$ is generalized Cohen--Macaulay.
\end{corollary}  
 
  	Suppose in addition to  Corollary \ref{1.4} that there is an open subset 
  	$U\subset\Spec(R)$ of codimension at least two such that it is  Gorenstein in codimension one. It follows  that $R$ is quasi--normal.
The recent result \cite[Theorem 1.1]{gl} says that $R$ is   weakly Arf  provided $R$ is 	$(S_2)$, Mori and $F$-pure.
As another application, we present two generalizations of \cite[Theorem 1.1]{gl}. Namely, we drop
the extra Mori assumption. Also, we  replace the $F$-pure rings  with  weakly normal rings  which is weaker. Consequently, we show:

\begin{corollary}
	Let $R$ be a  complete local domain of dimension $d\geq 2$ and prime characteristic  with $k=\overline{k}$. If  $R$ is generalized Cohen--Macaulay and weakly normal, then    $R$ is  weakly Arf.
\end{corollary}
  	
 Quasi-normal rings includes the class of Gorenstein rings.
 In particular,  these rings behave well under certain dualities.
 	The final section collects some duality  remarks on generalized Cohen--Macaulay modules over rings  satisfying  suitable $(G_n)$ conditions. These results are motivated by a result of Hartshorne,
  	and have some consequences. For instance:
  	
  	\begin{corollary}
  		Let $R$ be a $d$-dimensional complete Cohen--Macaulay local ring satisfying $(G_{d-1})$ and equipped with a canonical module $ \omega_R$. Then
 $\dim R/\fp=d$ for all $\fp\in\Ass( \omega_R\otimes  \omega_R)\setminus\{\fm\}$.
  	\end{corollary}
 From these duality results  we  present a connection to the symbolic powers, 	 see Corollary \ref{sym}. In the final section, we deal with the following   conjecture of Matlis:
 \begin{conjecture}\label{1.7}(See \cite[Page 51]{mat2})
  $\frac{\overline{R}}{R}$ is $h$-reduced.
 \end{conjecture}
Matlis answered  this over Q-rings, see \cite{q}.
Also, this was extended to a class of rings called $Q^n$-rings by Bertin, see \cite{ber}. We also partially  answer it, see Proposition \ref{sol}.


 	\section{(Higher) associated primes of $\frac{\overline{R}}{R}$}
Let $R$ be a   ring and let $S\subset R$ be the set of all elements which are not zero divisors in $R$.
We  localize $R$ at the multiplicative closed set $S$ to obtain the total quotient ring $Q(R):=S^{-1}R$. Recall that $\overline{R}$ stands for the
 	normalization  of $R$. 
 	By definition,  $\overline{R}$ is the integral closure of $R$ in $Q(R)$.
 	  We start by an example:
 	\begin{example}
 		Let $R:=k[x, y, z]/(y^2 -x^2(z^2 - x))$. Then $\Ass(\frac{\overline{R}}{R})=\{(x,y)\}$.
 		In particular,  $\Ass(\frac{\overline{R}}{R})^{\geq 2}=\emptyset$.
 	\end{example}

Recall that $R$ is a Mori ring if it is reduced and
  $\overline{R}$ is finite as an $R$-module.
Also, the notation $(-)^\ast$ stands for $\Hom_R(-,R)$.
The first item in the following when $R$ is \it{Mori} is in \cite{lv}.
Also, see \cite{m}.

\begin{proposition}\label{p}
Let $R$ be any ring. The following assertions hold: 
\begin{enumerate}
\item[i)] $\Ass(\frac{\overline{R}}{R})\subset\{\fp\in\Spec(R):\depth(R_{\fp})=1\}$;
\item[ii)]  if  $R$ is reduced, then  $\Ass(\frac{\overline{R}}{R})\subset\{\fp\in\Spec(R):\fp^{\ast\ast}=\fp\}$.
\end{enumerate}
\end{proposition}
	
	\begin{proof} i)
Let $\fp\in\Ass(\frac{\overline{R}}{R})$.	If $\fp$ is a minimal prime ideal of $R$. Then 
$R_{\fp}$ is artinian local ring, and the set of non-zerodivisor elements is just the invertible elements.
Thus, its  total ring of fractions is  $R_{\fp}$.  By definition, $R_{\fp}=\overline{R_{\fp}}$, i.e., 
$\fp\not\in\Supp(\frac{\overline{R}}{R})$. We may assume that $\Ht(\fp)>0$.	By definition, $\fp R_{\fp}\in\Ass(\overline{R}_{\fp}/R_{\fp})$.
	This implies that  $\Hom_{R_{\fp}}(R_{\fp}/\fp R_{\fp},\overline{R}_{\fp}/R_{\fp})\neq 0$.
	The long exact sequence induced by $0\to R_{\fp}\to \overline{R}_{\fp}\to \overline{R}_{\fp}/R_{\fp}\to 0$  shows that
	 $  0\neq\Hom_{R_{\fp}}(R_{\fp}/\fp R_{\fp},\overline{R}_{\fp}/R_{\fp})\subset\Ext_{R_{\fp}}^1(R_{\fp}/\fp R_{\fp},R_{\fp}).$ 
	We conclude that depth of $R_{\fp}$ is at most one.
Now we show that $\depth(R_{\fp})=1$. If this is not the case then we should have $\depth(R_{\fp})=0$.
In this case any elements of $\fp R_{\fp}$ is zerodivisor. Thus, $Q(R_{\fp})=R_{\fp}$. From, this $\frac{\overline{R_{\fp}}}{R_{\fp}}=0$,
i.e.,  $\fp\notin\Supp(\frac{\overline{R}}{R})$. This is a contradiction, because
$\Ass(-)\subset \Supp(-)$. In sum, $\depth(R_{\fp})=1$, as claimed.

ii) Let $\fp\in\Ass(\frac{\overline{R}}{R})$. Suppose on the way of contradiction
	that
	$\fp\subset\bigcup_{\fq\in\Ass(R)}\fq$. By prime avoidance, $\fp\subseteq \fq$ for some 
	$\fq\in\Ass(R)$. Since $R$ is reduced, $\Ass(R)=\min(R)$. Thus, $\fp=\fq$, e.g. $\fp R_{\fp}\in\Ass(R_{\fp})$.
	Therefore, $\depth(R_{\fp})=0$. By part i) we know $\depth(R_{\fp})=1$.
	This contradiction says that 	$\fp\nsubseteqq\bigcup_{\fq\in\Ass(R)}\fq$.
	In particular, $\fp$ contains a regular element and that  $\depth(R_{\fp})=1$.
	Under these assumptions, \cite[Corollary 2.3]{wol} indicates that $\fp$ is reflexive.
	\end{proof}
Recall that
Serre's $(S_n)$ condition  indicates 
$\depth(R_{\fp})\geq\min\{n,\Ht(\fp)\}$.
Serre's $(R_n)$ condition
means $R_{\fp}$ is regular for all prime ideal $\fp$ of height at most $n$.  He  characterized normality in terms of   $(S_2)$ and $(R_1)$, see \cite[Theorem 5.8.6]{EGA}.
\begin{definition}
	\begin{enumerate}
		\item[i)]
By  $(G_n)$ we mean  that $R_{\fp}$  is Gorenstein for any $\fp$ of height at most  $n$.	
\item[ii)] Recall from \cite{wol} that a ring $R$ is called quasi-normal  if it is  $(S_2)$ and $(G_1)$.
\end{enumerate}
\end{definition}

\begin{remark}
	More intrinsic, in a quasi--normal
 ring   principal
ideals have a unique representation as an intersection of irreducible ideals. To find a geometric picture, see \cite[Theorem 1.12]{har}. More interestingly, Hartshorne replaced $(R_1)$ with $(G_1)$
in order to recover a result of Max Noether who found the largest possible dimension of a linear system of given degree  
on a plane curve  of fixed degree.
\end{remark}

 \begin{example}\label{ob1}
	There is a $2$-dimensional quasi-normal local domain $(R,\fm)$  such that
	$\Spec(R)^1= \Ass(\frac{\overline{R}}{R})$.  In particular, $\Ass(\frac{\overline{R}}{R})$ is infinite.
\end{example}

\begin{proof}
	Let $R$ be any local integral domain such that for any non zero prime ideal $\fp$
	the ring $R_{\fp}$ is not normal. By \cite[Example 2.7]{nish} there is a 2-dimensional complete-intersection ring equipped with   that property. Since localization commutes with integral closure
	we deduce that $\Spec(R)\setminus\{0\}=\Supp(\frac{\overline{R}}{R})$. 
	\begin{enumerate}
		\item[Fact A):]  Let $ R$ be a Noetherian ring and  $M$ be any module. Then any $\fp\in\Supp(M)$ which is minimal among the elements of $\Supp(M)$ is an element of $\Ass(M)$.
	\end{enumerate}
	By applying Fact A) we see that 	every height one prime ideal is an associated prime of $\frac{\overline{R}}{R}$. 
	Since $\depth(R)=2$, and in view of  Proposition \ref{p}(i), we see $\fm\notin\Ass(\frac{\overline{R}}{R})$. So, $\Spec(R)^1= \Ass(\frac{\overline{R}}{R})$.
	
	To see the particular case, one may apply  Ratliff's weak existence theorem (see \cite[Theorem 31.2]{Mat})
	to find infinitely many height-one prime ideals.
\end{proof}

\begin{remark}
There is a 	  three-dimensional local domain $R$ 
with   infinity many prime ideals $\fp$
of height two such that $\depth(R_{\fp}) =1$, see \cite[Proposition 3.5]{fr}. In particular,  $\Ass^{\geq 2}(\frac{\overline{R}}{R})$
is infinite. 
\end{remark}

\begin{proposition}\label{pcr2}
	Let $R$ be a  ring which is $(S_2)$. Then $\Ass(\frac{\overline{R}}{R})^{\geq 2}=\{\fp\in\Spec(R):\fp^{\ast\ast}=\fp\}^{\geq 2}=\emptyset$.
\end{proposition}

\begin{proof}We may assume that $R$ is of dimension $d\geq 2$.
 Let $\fp\in\Ass(\frac{\overline{R}}{R})$. 
	In view of Proposition \ref{p}, $\depth(R_{\fp})=1$. Since $R$ is   $(S_2)$, $1=\depth(R_{\fp})\geq\min\{2,\Ht(\fp)\}$. Thus, $\Ht(\fp)=1$, and so $\Ass(\frac{\overline{R}}{R})^{\geq2}=\emptyset$.
Here, we prove $\{\fp\in\Spec(R):\fp^{\ast\ast}=\fp\}^{\geq 2}=\emptyset$. Suppose on the way  of contradiction that there is  a  reflexive prime ideal
	$\fp$ of height at least two.
	Since $R$ is   $(S_2)$, $\depth(R_{\fp})\geq \min\{2,\Ht(\fp)\}=2$. We know that the  $R_{\fp}$-module  $\fp R_{\fp}$
	is reflexive.
	Recall that for any $R_{\fp}$-module $(-)$ we have $\depth_{R_{\fp}}(\Hom_{R_{\fp}}(-,{R_{\fp}}))\geq 2$.
	We apply this to see $\depth_{R_{\fp}}(\fp R_{\fp})=\depth_{R_{\fp}}((\fp R_{\fp}^\ast)^\ast)\geq 2$.
	The exact sequence $ 0\to\fp R_{\fp}\to   R_{\fp}\to \frac{ R_{\fp}}{\fp R_{\fp}}\to 0$ 
	induces  $0=\HH^0_{\fp R_{\fp}}(R_{\fp})\to  \HH^0_{\fp R_{\fp}}(\frac{ R_{\fp}}{\fp R_{\fp}})\hookrightarrow  \HH^1_{\fp R_{\fp}}( \fp R_{\fp}).$ 
	Since $\HH^0_{\fp R_{\fp}}(\frac{ R_{\fp}}{\fp R_{\fp}})=\frac{ R_{\fp}}{\fp R_{\fp}}\neq 0$
	we deduce that the depth of  $ \fp R_{\fp}$ is at most one, which is a contradiction. So, 
	$\{\fp\in\Spec(R):\fp^{\ast\ast}=\fp\}^{\geq 2}=\emptyset$.
\end{proof}

\begin{theorem}\label{cr2}
	Let $R$ be a reduced ring of dimension $d\geq 2$. The following conditions are equivalent:
	\begin{enumerate}
		\item[i)] $R$ is $(S_2)$. 
		\item[ii)] $\{\fp\in\Spec(R):\fp^{\ast\ast}=\fp\}^{\geq 2}=\emptyset$.
		\item[iii)] $\Ass(\frac{\overline{R}}{R})^{\geq 2}=\emptyset$. 
		\end{enumerate}
In particular, if one of these hold, then
$\Spec(R)$ is connected
in codimension one.\end{theorem}

\begin{proof} $i)\Rightarrow ii)$: This is in Proposition \ref{pcr2}.

$ii)\Rightarrow iii)$: This is in  Proposition \ref{p}(ii).

$iii)\Rightarrow i)$: 
Assume  on the way of contradiction that $R$ is  not $(S_2)$.  
	Let $\fp\in\Spec(R)$   be such that $\depth(R_{\fp})<\min\{2,\Ht(\fp)\}$.
	By $(R_0)$ and $(S_1)$, we know  $\Ht(\fp)\geq 2$.
	Since $R$ is reduced, $\Ass(R)=\min(R)$. From this, $\depth(R_{\fp})\geq1$.
	In sum, 
	$1\leq\depth(R_{\fp})<\min\{2,\Ht(\fp)\}\leq 2$.
	This says that   $\depth(R_{\fp})=1$. Since $\depth(R_{\fp})=1$, there is $x/1\in \fp R_{\fp}$ such that $\depth(\frac{R_{\fp}}{x R_{\fp}})=0$.  	By definition, $\fp R_{\fp}\in\Ass(\frac{R}{(x)})_\fp$ and so $\fp\in\Ass(R/xR)$. Since $\Ht(\fp)\geq 2$, $\fp$ is not minimal over $xR$.  In the light of \cite[Theorem 2.6]{m} we see $\Ass(\frac{\overline{R}}{R})=\Sing^1(R)\cup E_1$
	where $E_1$ is the set of all embedded  primes of principal ideals.
So $\fp\in \Ass(\frac{\overline{R}}{R})^{\geq2}$, a contraction.

The particular case is in \cite[Corollary 2.4]{h}.
\end{proof}

The main application is as follows:

 \begin{theorem}\label{wv}
	Let $R$ be an integral domain  and suppose	a nonzero ideal  is reflexive
	if and only if all of its associated primes are of height one. Then  $R$ is  quasi-normal.
\end{theorem}

\begin{proof} There is nothing to prove if $d:=\dim(R)=0$. Then we may assume that $d>0$.
	First, we show $R$  is $(S_2)$. If $d=1$, this is trivial,
	because $R$ is Cohen--Macaulay. In the case $\dim(R)>1$, we know by the assumption that
	$$\{\fp\in\Spec(R):\fp^{\ast\ast}=\fp\}^{\geq 2}=\emptyset.$$
In view of Theorem \ref{cr2} we observe that $R$ is $(S_2)$. Thus, things are reduced to showing that $R$ is $(G_1)$.
Let $\fp$ be a height one prime ideal. The ring $A:=R_{\fp}$ is a $1$-dimensional Cohen--Macaulay local ring. Now, we recall the following result of Bass (see \cite[Theorem (6.2)]{bass}):	\begin{enumerate}
	\item[Fact A):]  Let $A$ be a $1$-dimensional Cohen--Macaulay local ring
	such that all of its ideals are reflexive. Then $A$ is Gorenstein.
\end{enumerate}
Let $J$ be any nonzero and proper ideal of $A$. Any ideal of $A$ is extended, i.e., there is an ideal $I\subset \fp$ of $R$ such that
$IA=J$. Since $I$ is nonzero and height of $\fp$ is one, we deduce that $\fp\in\min(I)$. Let $I=\cap_{i=1}^n \fq_i$ be any minimal primary decomposition  of $I$ and set $\fp_i:=\rad(\fq_i)$ which is a prime ideal. We may assume that $\rad(\fq_1)=\fp$.  By definition, $\fq_1$ is a  $\fp$-primary  ideal. If $\fp_i\subset  \fp$, then $\fp=\fp_i$.
 Set $S:=R\setminus \fp$, and note that  $\fp_i\cap S\neq \emptyset$ for all $i\neq 1$.  It turns out that  $J=S^{-1}I=\cap_{i=1}^n S^{-1}\fq_i=S^{-1}\fq_1$  is a minimal primary decomposition, because $S\cap\fp_i\neq \emptyset$ for all $i>1$. After replacing $I$ with $\fq_1$ we may assume in addition that 
$I$ is $\fp$-primary. In particular, all of its associated prime ideals are of height one.
By the assumption, $I$ is reflexive. Any localization of a reflexive module is reflexive. So, $J$ is reflexive. We proved that
any ideal of $A$ is reflexive. In the light of Fact A) we deduce that $A$ is Gorenstein. By definition, $(G_1)$ and $(S_2)$ mean the ring is quasi-normal.
\end{proof}



\begin{fact}(Matsumura--Ogoma)\label{mat}
 If $R$ is reduced, then 	 $\Ass(\frac{\overline{R}}{R})^{\geq2}=\{\fp\in\Spec(R):\depth(R_{\fp})=1\}^{\geq2}$.
 	\end{fact}

 \begin{proof} Take $\fp$ be of height greater than one such that $\depth(R_{\fp})=1$.
 By the above proof there is $x/1\in \fp R_{\fp}$ such that   $\fp\in\Ass(R/xR)$, i.e.,
 $\fp\in E_1\subset\Ass(\frac{\overline{R}}{R})^{\geq2}$. The reverse inclusion
 is in Proposition \ref{p}.
 \end{proof}

Here, we present a connection to  $(S_2)$ locus:

\begin{corollary}(Grothendieck)\label{gr}
	Let $R$ be excellent. Then $\{\fp\in\Spec(R):R_{\fp} \emph{ is }  (S_2)\}$ is open.
\end{corollary}

\begin{proof}
 Let $\fp\in\Spec(R)$.	By Nagata's criterion, it is enough to prove $\{P\in\Spec(\frac{R}{ \fp})|	(\frac{R}{ \fp})_P
	\emph{ is }  (S_2)\}$  contains a nonempty  open set. Recall that every homomorphic image of an excellent ring,   is again excellent. Then, without loss of generality we may and do assume that $R$ is an integral domain.	
Since $\overline{R}$ is finite over,  $\Ass(\frac{\overline{R}}{R})^{\geq 2}\subset \Ass(\frac{\overline{R}}{R})$
	is finite. Let $\{\fp_1,\ldots,\fp_n\}=\Ass(\frac{\overline{R}}{R})^{\geq 2}$. Define $J:=\bigcap_{i=1}^n \fp_i$. Since $R$ is domain, $J\neq 0$ and so $D:=\Spec(R)\setminus\V(J)$
	is nonempty and open. It remains to prove: 	\begin{enumerate}
		\item[Claim:] $D\subset \{\fp:R_{\fp} \emph{ is }  (S_2)\}$.
	\end{enumerate} If not,
	then there is $\fp\in D$ such that  $0<\depth(R_{\fp})<\min\{2,\Ht(\fp)\} \leq 2,$ i.e.,
	$\depth(R_{\fp})=1$ and $\Ht(\fp)\geq 2$. By the above corollary,  $\fp\in \Ass(\frac{\overline{R}}{R})^{\geq2}=\{\fp_i:1 \leq i \leq n\}$. Let
  $i$ be  such that $\fp=\fp_i\supseteq J$, i.e., $\fp\in\V(J)$. This  contradiction yields the claim.
	\end{proof}
The Cohen--Macaulay locus of $R$ is $\CM(R):=\{\fp\in\Spec(R):R_{\fp}\textit{ is Cohen--Macaulay}\}$.

\begin{lemma}\label{red}
Let $R$  be reduced. Then $\Spec(R)^2\setminus\CM(R)\subseteq\Ass(\frac{\overline{R}}{R})^{\geq2}.$
\end{lemma}

\begin{proof}
 Let $\fp\in \Spec(R)^2\setminus\CM(R)$. Reduced rings are $(S_1)$.
	From this, $\depth(R_{\fp})>0$.
	Since $R_{\fp}$ is not Cohen--Macaulay and $\Ht({\fp})\geq 2$ we deuce that  $\depth(R_{\fp})\leq1$.
	So,   $\depth(R_{\fp})=1$.  
	By Fact \ref{mat}, $\fp\in \Ass(\frac{\overline{R}}{R})$.
\end{proof}
It may be  natural to find situations for which   the above inclusion becomes an equality.
To this end, a ring $R$ is called   almost Cohen--Macaulay   if $\gr(\fp, R) = \gr(\fp R_{\fp} , R_{\fp} )$ for every $\fp
\in\Spec(R)$, see \cite{kan}.

\begin{proposition}\label{acm}
	Let $R$ be  almost Cohen--Macaulay and of dimension $d\geq 2$. The following holds: \begin{enumerate}
		\item[i)] $\Ass(\frac{\overline{R}}{R})^{\geq2}\subseteq\Spec(R)^2\setminus\CM(R). $	
		\item[ii)] Suppose in addition $R$ is reduced then $\Ass(\frac{\overline{R}}{R})^{\geq2}=\Spec(R)^2\setminus\CM(R). $
		In particular,
		$\Ass(\frac{\overline{R}}{R})^{\geq2}$ is finite provided  $\CM(R)$ is open.
	\end{enumerate}   
\end{proposition}

\begin{proof}
i) Let $\fp\in	\Ass(\frac{\overline{R}}{R})^{\geq2}$. By definition, $\Ht(\fp)\geq 2$. In view of Proposition \ref{p}, $\depth(R_{\fp})\leq1$. Since $R$  is almost Cohen--Macaulay we know that $\Ht(\fp)\leq \depth(R_{\fp})+1$. We plug this in the previous
inequalities  $2\leq\Ht(\fp)\leq \depth(R_{\fp})+1\leq2.$ From this, $\Ht(\fp)=2$. Thus,
$\fp\in \Spec(R)^2\setminus\CM(R)$.

ii) The first claim is in Lemma \ref{red}.
Now, suppose $\CM(R)=\Spec(R)\setminus\V(J)$ for some ideal $J$. Since $R$ is $(S_1)$, we see
that $\Ht(J)\geq 2$. It turns out that $\Spec(R)^2\setminus\CM(R)$ is finite, and the claim follows.
\end{proof}

	\begin{example}\label{ex3}
i)	Let $R:=k[[x^4,x^3y,xy^3,y^4]]$. Then $\Ass(\frac{\overline{R}}{R})^{\geq2}= \{\fm\}$.

ii)		
	Let $(R,\fm)$ be a $3$-dimension complete  reduced local ring of depth one. Then   $\Ass(\frac{\overline{R}}{R})^{\geq2}=\Spec(R)^{\geq2}\setminus\CM(R).$
\end{example}

\begin{proof}i) This  is trivial, because $2$-dimensional reduced rings are almost  Cohen--Macaulay.
	
	ii) The assumption implies that
	 $\CM(R)=\Spec(R)\setminus\V(J)$ for some ideal $J$. The height of $J$ should be two, because $R$ is reduced. Let $\{\fp_1,\ldots,\fp_t\}$
	 be the set of all minimal prime ideals of $J$. Note that each $\fp_i$ is of height two.
	 Let $\fp\in\Spec(R)$ be of height two. Then $R_{\fp}$ is Cohen--Macaulay if and only if $\fp\notin \min(J)$. From this, $\depth(R_{\fp_i})=1$. Thus, 
$\Ass(\frac{\overline{R}}{R})^{\geq2}=\{\fp_i\}_{i=1}^t\cup\{\fm\}.$ 
This is equal to $\Spec(R)^{\geq2}\setminus\CM(R).$
\end{proof}

\begin{remark}
Adopt the notation of Example \ref{ex3}(ii), and write $R=\frac{S}{I}$ where $S$ is regular of dimension $n$ and $I$
	is of height $n-3$. Let $\frak a_i:=\Ann_S(\Ext_S^{n+i-4}(R,S))$ where $1\leq i\leq 4$
	and set $J:=\frac{\frak a_2\frak a_3\frak a_4+\frak a_1\frak a_3\frak a_4+\frak a_1\frak a_2\frak a_3}{I}$. Then  $\Ass(\frac{\overline{R}}{R})^{\geq2}=\Min(J)\cup\{\fm\}.$ 
\end{remark}

A local ring $(R,\fm)$ is called  generalized Cohen--Macaulay if $\ell(\HH^{i}_{\fm}(R))<\infty$ for all $i<\dim R$.
As an example, the local ring at the vertex of an affine cone over a smooth projective variety is generalized Cohen--Macaulay.
For more details, see \cite{sch}.
\begin{proposition}\label{ge}
	Let  $(R,\fm)$  be a  complete local ring of dimension $d\geq 2$. If  $R$ is generalized Cohen--Macaulay, then  $\Ass(\frac{\overline{R}}{R})^{\geq2}\subseteq \{\fm\}.$  Suppose
	in addition $R$ is reduced. Then $\Ass(\frac{\overline{R}}{R})^{\geq2}= \{\fm\}$
if and only if $\depth(R)=1$. 
\end{proposition}

\begin{proof}The assumptions grantee that $R$ is   Cohen--Macaulay over the punctured spectrum (see \cite{sch}).
	Let $\fp\in	\Ass(\frac{\overline{R}}{R})^{\geq2}\setminus\{\fm\}$. By definition, $\Ht(\fp)\geq 2$. In view of Proposition \ref{p}, $\depth(R_{\fp})\leq1$. Since $R$  is  Cohen--
	Macaulay over the punctured spectrum, we know that $2\leq \Ht(\fp)=\depth(R_{\fp})\leq1$. This contradiction says that $\Ass(\frac{\overline{R}}{R})^{\geq2}\setminus\{\fm\}=\emptyset$.
	In other words, $\Ass(\frac{\overline{R}}{R})^{\geq2}\subset \{\fm\}$. 
	
Here, we assume that $R$ is reduced and we compute $\Ass(\frac{\overline{R}}{R})^{\geq2}$. 
We have two possibilities: i)   $\depth(R)\neq 1$, or ii) $\depth(R)= 1$.

i)  Suppose $\depth(R)\neq 1$. Then 	$\Ass(\frac{\overline{R}}{R})^{\geq2}=\emptyset.$
	Indeed, 
	we apply Proposition \ref{p}(i) to see that $\fm\notin\Ass(\frac{\overline{R}}{R})^{\geq2}.$ We combine
	this with the first paragraph to see $\Ass(\frac{\overline{R}}{R})^{\geq2}=\emptyset.$

ii)   Suppose $\depth(R)= 1$. Then $\Ass(\frac{\overline{R}}{R})^{\geq2}=\{\fm\}.$	Indeed, we know that  $R$ is not $(S_2)$. We apply Theorem \ref{cr2} to see
	 $\Ass(\frac{\overline{R}}{R})^{\geq2}\neq\emptyset.$  We incorporate this
		with $\Ass(\frac{\overline{R}}{R})^{\geq2}\subseteq \{\fm\}$ to conclude that $\Ass(\frac{\overline{R}}{R})^{\geq2}=\{\fm\}.$
\end{proof}

	An ideal $I$ of a local ring $A$ is called Cohen--Macaulay
	if $A/I$  is a Cohen--Macaulay ring.

\begin{observation}  \label{remint} 
Let $I$ and $J$ be Cohen--Macaulay  ideals of a local ring $(A,\fm)$
 such that $I+J$ is $\fm$-primary.  Set $R:=A/I\cap J$. The following holds:
		\begin{enumerate}
		\item[i)] If $\dim(A/I)=\dim(A/J)\geq 2$ and $J$, then $R$ is generalized Cohen--Macaulay.
		\item[ii)] If $2\leq \dim(A/I)<\dim(A/J)$ and $J$, then $R$ is not generalized Cohen--Macaulay.
	\end{enumerate}
In each cases, $\Ass(\frac{\overline{R}}{R})^{\geq2}=\{\fm\} $ provided $I$ and $J$ are radical. 
\end{observation}

\begin{proof}i) We have $\min(I\cap J)\subset\Min(I )\cup\Min(J)$.
	Since Cohen--Macaulay rings are equidimensional, we have $d:=\dim(R)=\dim(A/I)=\dim(A/J)$.
	Use Grothendieck vanishing theorem along with the long exact sequence of local cohomology modules induced from the exact sequence $0\to \frac{A}{I\cap J}\to \frac{A}{I}\oplus\frac{A}{J}\to\frac{A}{I+J}\to 0$ to deduce that
	$\HH^1_{\fm}(R)\cong \frac{A}{I+J}$ and that $\HH^i_{\fm}(R)=0$ for all $i\neq 1,d$. Thus,
$R$ is generalized Cohen--Macaulay and $\depth(R)=1$. 
Suppose now that $I$ and $J$ are radical. Then $R$ is reduced. From this, and in view of
Proposition \ref{ge} we conclude that $\Ass(\frac{\overline{R}}{R})^{\geq2}= \{\fm\}$.

ii) It is easy to see that $R$ is not generalized Cohen--Macaulay, because it
is not equidimensional. Similar  to i) 	$\HH^1_{\fm}(R)\neq 0$ and so $\depth(R)=1$.
Let $\fp:=\frac{P}{I\cap J}$ be any prime ideal of $R$ different from $\fm$.
Then $P\supseteq I$ or $P\supseteq J$. If $P$ contain both of them, as $\rad{I+J}=\fm$, we see $P=\fm$
which is excluded. By symmetry, we assume that $P\nsupseteq I$.
 Recall that
 localization commutes intersections and that $IA_p= A_P$. So, $R_{\fp}=A_{P}/JA_{P}$ which is Cohen--Macaulay. 
 In sum, $R$ is Cohen--Macaulay over the punctured spectrum. In particular,
 there is no prime ideal $\fp\in\Spec(R)\setminus\{\fm\}$ of height bigger than one such that
$\Ass(R_{\fp})=1$, i.e., $\Ass(\frac{\overline{R}}{R})^{\geq2}\subset\{\fm\} $. 
Suppose now that $I$ and $J$ are radical. Then, $R$ is reduced and is not $(S_2)$.
 In view of Theorem \ref{cr2} $\emptyset\neq\Ass(\frac{\overline{R}}{R})^{\geq2}\subset\{\fm\} $. So,
$\Ass(\frac{\overline{R}}{R})^{\geq2}=\{\fm\} $.
\end{proof}

This cannot be extended to the higher intersections  (also, shows the importance of $I+J$ is $\fm$-primary):

\begin{example}\label{inters}
Let	$R:=\frac{\mathbb{Q}[[x_1,\ldots,x_6]]}{(x_1,x_2)\cap(x_3,x_4)\cap(x_5,x_6)}$.
	The following holds: \begin{enumerate}
		\item[i)] $\{(x_1,x_2,x_3,x_4),(x_1,x_2,x_5,x_6),(x_3,x_4,x_5,x_6)\}\subset\Ass(\frac{\overline{R}}{R})^{\geq2} $ and
		\item[ii)] $\fm\notin\Ass(\frac{\overline{R}}{R})^{\geq2}$.\end{enumerate}
\end{example}

\begin{proof}
Let $\fq:=(x_1,x_2,x_3,x_4)$ and $A:=\mathbb{Q}[[ x_1,\ldots,x_6]]$. Recall that
 localization commutes intersections and that $(x_5,x_6)A_{\fq}=A_{\fq}$. 
 From these, $R_{\fq} =\frac{A_{\fq}}{(x_1,x_2) A_{\fq}\cap(x_3,x_4) A_{\fq}}$ which is a 2-dimensional ring and of depth one (apply Remark \ref{remint} for $A_{\fq}$).
 In view of Fact \ref{mat}, $\fq\in \Ass(\frac{\overline{R}}{R})^{\geq2}$.
 Similarly, $\{(x_1,x_2,x_5,x_6),(x_3,x_4,x_5,x_6)\}\subset\Ass(\frac{\overline{R}}{R})^{\geq2}$.
 The projective dimension of $R$ as an $A$-module is four. By Auslander--Buchsbaum formula, $\depth(R)=2$.  In view of  Proposition \ref{p}(i), $\fm\notin \Ass(\frac{\overline{R}}{R})^{\geq2}$.
\end{proof}

\begin{example}  \label{pr} 
	Let $I$ and $J$ be Cohen--Macaulay  ideals of a local ring $(A,\fm)$
	such that $I+J$ is $\fm$-primary.  Set $R:=A/I  J$. Then, $\Ass(\frac{\overline{R}}{R})^{\geq2}\subset\{\fm\} $. Indeed, it is easy to see that $R$ is Cohen--Macaulay  over the punctured spectrum. From this, $\Ass(\frac{\overline{R}}{R})^{\geq2}\subset\{\fm\} $.
\end{example}

	\section{weak versions of normality  and being $(S_2)$}

An integral extension of rings $i : R \hookrightarrow S$ is said to be
subintegral (resp. weakly subintegral) it induces a bijection on the prime spectrum, and 
the induced maps on the residue fields, $k(i^{-1}(\fp))\to k(\fp)$ is an isomorphism (resp. a purely inseparable field extension). 

\begin{definition}
	Let $^+R$ (resp. $^{\ast}R$)
  be the  unique largest subextension of $R$ in $\overline{R}$ such that  $R\subset (^+R)$ (resp. $R\subset (^\ast R)$) is subintegral (resp. weakly subintegral). A ring
	$R$ is said to be seminormal (resp. weakly normal) if $R=(^+R)$ (resp. $R=(^{\ast}R)$).
\end{definition}

\begin{fact}\label{f2}
	Let $(R, \fm)$ be a reduced local ring of characteristic $p$ equipped with  the Frobenius map $F$.
\begin{enumerate} 
\item[i)](See  \cite{du}) Assume $R$ is
 weakly normal on its punctured spectrum. Then  $R$ is weakly normal if and only if the action of $F$ on $\HH^1_{\fm}(R)$  is injective.
\item[ii)] (See \cite[Theorem 3.1]{sw})	
Assume $(R, \fm)$  is a  strictly Henselian local domain, homomorphic image of a Gorenstein domain and
	$\dim R \geq 2$. Then $\HH^1_{\fm}(R)$ is $F$-torsion, i.e.,
	each of its elements annihilated after the action of some Frobenius powers.
 \end{enumerate}
\end{fact}

\begin{lemma}\label{first}
	Let $(R, \fm,k)$  be a  complete and weakly normal domain of characteristic $p$ where $k$ is  separably closed. The following assertions are hold:
\begin{enumerate} 
	\item[i)]	One has $\depth(R)\geq\min\{2,\dim R\}$. In particular,
	the punctured spectrum is connected.
		\item[ii)] Assume in addition	that $\dim(R)\geq 2$, let $\mathcal{X}:=\Spec(R)$ and $\mathcal{X}^\circ:=\mathcal{X}\setminus\{ \fm \}$ be the punctured spectrum.
	Then  the restriction map $f:\HH^0(\mathcal{X},\mathcal{O}_{\mathcal{X}})\to \HH^0(\mathcal{X}^\circ,\mathcal{O}_{\mathcal{X}^\circ})$ is an isomorphism.
\end{enumerate}
\end{lemma}

\begin{proof}
i)	We may assume that $d:=\dim R\geq 2$, and we  show that $\HH^1_{\fm}(R)=0$. This follows by showing that the action of Frobenious
	on $\HH^1_{\fm}(R)$ is  both injective and torsion. These are done by Fact  \ref{f2}.
	In order to apply Fact  \ref{f2}(i), we recall that  a ring is weakly normal if and only if all of its localizations at prime
	ideals are weakly normal.
In order to prove the particular case, first assume that  $d= 0$ (resp. $d= 1$). Then, the punctured spectrum is empty (resp. singleton), and so
		connected. Then we may assume that $d>1$. We proved that $\depth(R)\geq\min\{2,d\}\geq 2$. Now, the claim follows from \cite[Proposition 2.1]{h}.

		ii) 	There is an exact sequence
		$0\to\HH^0_{\fm}(R)\to\HH^0(\mathcal{X},\mathcal{O}_{\mathcal{X}})\stackrel{f}\lo \HH^0(\mathcal{X}^\circ,\mathcal{O}_{\mathcal{X}^\circ})\to\HH^1_{\fm}(R)\to0.$	
		By the first item, $ \HH^0_{\fm}(R)=\HH^1_{\fm}(R)=0$. So,  $f$ is an isomorphism.
\end{proof}

\begin{proposition}\label{cge}
	Let 	$(R, \fm,k)$  be a  complete local domain of dimension $d\geq 2$ of prime characteristic where $k$ is  separably closed. The following assertions are hold:
	\begin{enumerate} 
		\item[i)]If  $R$ is generalized Cohen--Macaulay and weakly normal, then    $\Ass(\frac{\overline{R}}{R})^{\geq2}=\emptyset$. In particular, 	$R$ is   $(S_2)$.
			\item[ii)] Let $I$ be an ideal such that $\dim V (I)\leq \dim R-2$, 
			then $R=\Gamma (\Spec R\setminus\V (I), R)$.\end{enumerate}
\end{proposition}

\begin{proof}
i)	In view of  Lemma \ref{first} we see $\depth(R)\geq\min\{2,\dim R\}\geq 2$. 
 By Proposition  \ref{ge},  $\Ass(\frac{\overline{R}}{R})^{\geq2}\subset \{\fm\}$. Due to  Proposition 
	\ref{p}  we know  $\Ass(\frac{\overline{R}}{R})^{\geq2}=\emptyset$. We incorporate this along with Theorem \ref{cr2} to observe that   $R$ is   $(S_2)$.
	
	ii) This follows from i) by applying a well-known trick.
\end{proof}

 \begin{corollary}
 	Let 	$(R, \fm,k)$  be   complete, weakly normal, generalized   Cohen--Macaulay, of  prime characteristic $p$  of dimension $d>1$ and $k=\overline{k}$. Suppose there is an open subset 
 	$U\subset\Spec(R)$ such that for any $x\in U$ of $\codim\leq 1$, $\mathcal{O}_x$ is quasi-Gorenstein. 
 	The following assertions are true:
 	\begin{enumerate}
 		\item[i)] For any $\fp$ of height $1$, the ring $R_{\fp}$ is Gorenstein. In particular, $R$ is quasi-normal.
 		\item[ii)] Suppose in addition  $p\neq 2$. Then, after possible shrinking of $U$,  we have $\widehat{\mathcal{O}_x}\cong\frac{k[[X_1,\ldots,X_n]]}{(X_1X_2)}$ for any closed and singular point $x\in U$.
 	\end{enumerate}
 \end{corollary}

\begin{proof} In the light of Proposition \ref{cge} we see that  $R$ is   $(S_2)$. 
	By definition any quasi-Gorenstein and Cohen--Macaulay  ring is  Gorenstein.
Thus, for any $x\in U$ of $\codim\leq 1$ the ring  $\mathcal{O}_x$ is Gorenstein. 	
	Also, weakly normal rings are seminormal. Therefore,  i)  is in \cite[Proposition 9.3]{gt}, and ii)   is in \cite[Theorem 9.10]{gt}.
\end{proof}

\begin{definition}
Suppose $ x$ is any non-zerodivisor on $R$ and let  $y, z \in R$ be any pair such that  $\{y/x, z/x \}\subset \overline{R}$.
Recall from \cite{arf} that   $R$ is called \it{weakly Arf} if $yz/x \in R$. \end{definition}

\begin{proposition}\label{wa}
	Let  $R$  be   
	$(S_2)$ and of prime characteristic $p$. Adopt one of the following assumptions:
	\begin{enumerate}
		\item[i)]  $R$ is $F$-pure, or
		\item[ii)]  $R$ is weakly normal and Mori.
	\end{enumerate} Then $R$ is  weakly Arf.
\end{proposition}

\begin{proof}By \cite[Theorem 2.6]{arf} it is enough to show that $R_{\fp}$ is   weakly
	Arf   for every $\fp$ with $\depth(R_{\fp})\leq 1$.
 Due to the $(R_0)$ and $(S_2)$ conditions, we may assume that
 $(R,\fm)$	is a one-dimensional Cohen--Macaulay local ring. Note that
	the localization of weakly normal is again weakly normal. So, $R$ is weakly normal.
	Weakly normal rings are reduce.	Over reduced Mori rings of prime characteristic $p$
	there is nice characterization of weakly normal rings: If $a\in \overline{R}$ and $a^p\in R$, then $a\in R$.
	This condition
	implies that the  assignment $x+R\in Q(R)/R\mapsto x^p+R$ induces an injection $f:Q(R)/R \hookrightarrow Q(R)/R$.
	Similarly, in the $F$-pure case,  we know  $f:Q(R)/R \hookrightarrow Q(R)/R$ is injective.
	Indeed, since $R$ is reduced,  $F:R\to R$ is injective. Tensor it with $Q(R)/R $
	we get that $F\otimes 1:Q(R)/R \hookrightarrow Q(R)/R$ is injective.

i) 	Another use of $(R_0)$ shows that $\Supp(\frac{\overline{R}}{R})\subset\{\fm\}$. Let $x+R\in \frac{\overline{R}}{R}$.
	Then $\ell((x+R)R)< \infty$, i.e., $\fm^t(x+R)R=0$ for all $t\gg 0$.  Hence $f^t(\fm((x+R)R)) =0$ for all   $t\gg 0$. The injectivity of $f$
	implies that $\fm((x+R)R)= 0$. Thus,  $\fm\frac{\overline{R}}{R}= 0$. 
	Now, let $x, y$ and $z \in R$ be  such that $ x$ is a non-zerodivisor on $R$ and that  $y/x, z/x \in \overline{R}$. 
	Then $\frac{ yz}{x}= x
	\frac{y}{x}\frac{z}{x}\in\fm \overline{R}\subset R$. By definition, $R$ is   weakly Arf.

ii)  This is in the above proof.
\end{proof}


\begin{corollary}\label{obcge}
	Let $R$ be a  complete local domain of dimension $d\geq 2$ and prime characteristic with separably closed residue field. 
	If  $R$ is generalized Cohen--Macaulay and weakly normal, then    $R$ is  weakly Arf.
\end{corollary}

\begin{proof}
	By Proposition \ref{cge},  $R$ is   $(S_2)$. So, the claim is in Proposition \ref{wa}.
\end{proof}


\begin{example}
Let $R:=k[[t^4,t^5,t^6]]$. Then $R$ is a quasi-normal local domain which is not weakly Arf.
\end{example}
\begin{proof}
As $R$ is Gorenstein, it is quasi-normal. Since $t=\frac{t^6}{t^5}\in Q(R)$,
we  deduce that $t\in\overline{R}$. Set $x:=t^5$ and $y:=z:=t^6$. Note that
$\frac{y}{x}=\frac{z}{x}=t\in\overline{R}$. Also, $\frac{yz}{x}=\frac{t^{12}}{t^5}=t^7\notin R$.
Thus,  $R$  is not weakly Arf.\end{proof}

\begin{corollary}\label{obcge}
	Let $R$ be a  quasi-normal and seminormal local Mori ring. Then   $R$ is  weakly Arf.
\end{corollary}

\begin{proof}
Similar to Proposition \ref{wa}, without loss of generality we may assume that $R$ is a  one-dimensional and local. In addition, we can assume that $R$ is Gorenstein, seminormal and Mori. There
is nothing to prove if $R$ is regular. Then we may assume that $R$ is not normal. In particular,
we are in the situation of 
\cite[Theorem 8.1]{gt} to deduce that  $e(R)=2$. By \cite[Corollary 2.8]{arf},  $R$ is  weakly Arf.
\end{proof}

\section{ Duality  and $(G_{n})$}

We start by simplifying the following result of Hartshorne:

\begin{fact}\label{har}(See \cite[Corollary 1.14]{har})
Let $R$ be a local Gorenstein ring, and let $M$ be a finitely
generated reflexive $R$-module. Then $M$ is maximal Cohen--Macaulay if and only if $M^\ast$ is maximal
Cohen--Macaulay.	
	\end{fact}

\begin{proof}
	By the Auslander-Bridger formula \cite{AB}, a module is maximal Cohen--Macaulay if and only if
	it is totally reflexive. By definition, a reflexive module is totally
	reflexive iff its dual is totally reflexive.
\end{proof}

\begin{example}
	The reflexivity of $M$ is important. To this end, let $(R,\fm)$ be any Cohen--Macaulay ring of dimension at least two. Then $\fm$ is not maximal Cohen--Macaulay (its depth is one). But, $\fm^\ast\cong R$
	is   maximal Cohen--Macaulay.
\end{example}


\begin{lemma}\label{ll}
	Let $R$ be a $d$-dimensional Cohen--Macaulay local ring with a canonical module $ \omega_R$.
	Let $M$ be  locally maximal Cohen--Macaulay on $\Spec(R)\setminus\{\fm\}$. Then, 
 $\HH^i_{\fm}(M\otimes_R \omega_R)\cong \Hom ( \HH^{d + 1 - i}_{\fm} ( M^\ast),E_R(k))$
for all $2 \leq i \leq d - 1$.
 \end{lemma}

\begin{proof}
	This  is a straightforward modification of \cite[Proposition 1.13]{har}, and we leave it to the reader.
\end{proof}

\begin{corollary}\label{DG}
	Let $R$ be a local Gorenstein ring, and let $M$ be a finitely
	generated  $R$-module of dimension $d:=\dim R$. If $M$ is generalized Cohen--Macaulay (resp. quasi-Buchsbaum) then  $M^\ast$ is  generalized 
	Cohen--Macaulay (resp. quasi-Buchsbaum).	
	The converse holds if  $M$ is reflexive.
\end{corollary}

\begin{proof}First, we deal with the case that  $M$ is generalized Cohen--Macaulay.
	We may assume that $d>1$ and we are going to show $\ell(\HH^{i}_{\fm} ( M^\ast))<\infty$.
	Recall that $M$ is locally maximal Cohen--Macaulay on $\Spec(R)\setminus\{\fm\}$.
 We Have $\HH^0_{\fm}( M^\ast)=\HH^1_{\fm}( M^\ast)=0$, because $\depth( M^\ast)\geq 2$. Let $2\leq i\leq d$.
Since Matlis duality does
not change the length, and in view of Lemma \ref{ll} we deduce that $\ell(\HH^{i}_{\fm} ( M^\ast))<\infty$. By repeating this argument, we have $\ell(\HH^{i}_{\fm} ( M^{\ast\ast}))<\infty$. 
Suppose $M$ is reflexive. Then we have $\ell(\HH^{i}_{\fm} ( M))<\infty$.
Now, recall that a module is called quasi-Buchsbaum if $\fm\HH^i_{\fm}(-)=0$ for all $i<\dim(-)$. As Matlis duality does
not change the annihilators, and in view of Lemma \ref{ll}, we get the claim.
\end{proof}

\begin{corollary}
	Suppose in addition to Corollary \ref{DG}  that
	$\HH^{i}_{\fm} ( M)=0$ for all $i\ne \depth(M),\dim(M)$.
	Then $M^\ast$ is  Buchsbaum if $M$ is  Buchsbaum.		The converse holds if  $M$ is reflexive.
\end{corollary}

\begin{proof}
	 Combine  \cite[Proposition I.2.12]{bus} with Corollary \ref{DG}.
\end{proof}

\begin{corollary}\label{ext}
	Let $R$ be a complete local Gorenstein ring of dimension $d$. If $M$ is generalized Cohen--Macaulay and  of dimension  $t$,
	then  $\Ext^{d-t}_R(M,R)$ is as well.
\end{corollary}

\begin{proof}We may assume that  $t>0$.
	The fact that $\Ext^{d-t}_R(M,R)$ is $t$-dimensional is due to Grothendieck, see \cite[6.4.4]{41}.
	Since   $M$ is generalized Cohen--Macaulay, $M$ is locally   Cohen--Macaulay on $\Spec(R)\setminus\{\fm\}$. It follows that
	$\Ext^{d-t}_R(M,R)$ is locally Cohen-Macaulay. Let  $\fp\in\Ass(\Ext^{d-t}_R(M,R))\setminus\{\fm\}$.
	There is $P\in\Supp(\Ext^{d-t}_R(M,R))$ of height $d-1$ which contains $\fp$.
	Localization at $P$ shows that $\fp R_P\in\Ass(\Ext^{d-t}_{R_P}(M_P,R_P))$.
	In particular, $\Ext^{d-t}_{R_P}(M_P,R_P)$ is a nonzero Cohen--Macaulay module and of dimension $t-1$.
	In particular, it is equidimensional. Thus, $\dim(\frac{R_P}{\fp R_P})=t-1$. From this,
	$\dim(R/\fp)\geq (t-1)+1=t$. 
	Also, the reverse inequality is true, because  $\Ext^{d-t}_R(M,R)$ is $t$-dimensional. These properties characterizes the generalized Cohen--Macaulay property, see \cite{sch}. Thus, $\Ext^{d-t}_R(M,R)$ is generalized Cohen--Macaulay.
\end{proof}
\begin{corollary}\label{ext}
	Let $R$ be a $d$-dimensional complete Cohen--Macaulay local ring  equipped with a canonical module $ \omega_R$. If $M$ is generalized Cohen--Macaulay of dimension  $t$,
	then  $\Ext^{d-t}_R(M, \omega_R)$ is  generalized 
	Cohen--Macaulay.	
\end{corollary}

\begin{corollary}\label{o2}
		Let $R$ be a $d$-dimensional complete Cohen--Macaulay local ring satisfying $(G_{d-1})$ and equipped with a canonical module $ \omega_R$. Then
		$ \omega_R\otimes  \omega_R$ is generalized Cohen--Macaulay. In
		particular, $\dim R/\fp=\dim R$ for all $\fp\in\Ass( \omega_R\otimes  \omega_R)\setminus\{\fm\}$.
\end{corollary}

\begin{proof}Without loss of generality we may assume that $d\geq 2$.
We know that $ \omega_R^\ast$ is locally free over the punctured spectrum
and that $\Ass( \omega_R^\ast)	=\Ass(\Hom_R( \omega_R,R))=\Supp( \omega_R)\cap\Ass(R)
=\Ass(R)$. From this, and in view of see \cite{sch}, we know that $ \omega_R^\ast$ is 	generalized Cohen--Macaulay.
In view of Lemma \ref{ll} we deduce that $\ell(\HH^{i}_{\fm} ( \omega_R\otimes  \omega_R))<\infty$
for all $2\leq i\leq d$. Now we compute  $\HH^{1}_{\fm} ( \omega_R\otimes  \omega_R)$. Since $R$ is generically Gorenstein, the canonical module
can be identified with an ideal. In particular, there is an exact sequence $0\to  \omega_R \to R\to R/ \omega_R\to 0$.
We drive the following exact sequence$$0\lo \Tor_1^R(R/ \omega_R, \omega_R)\lo  \omega_R\otimes_R  \omega_R\lo  \omega_R\lo  \omega_R\otimes_R R/ \omega_R\lo 0\quad(\ast)$$
We break down
$(\ast)$ into
a) $0\to \Tor_1^R(R/ \omega_R, \omega_R)\to  \omega_R\otimes_R  \omega_R \to L\to 0$ and
b) $0\to L\to   \omega_R\to   \omega_R/ \omega_R^2\to 0$.
Since $ \omega_R$ is locally free, $\Tor_1^R(R/ \omega_R, \omega_R)$ is of finite length. We conclude   from Grothendieck
vanishing theorem that $\HH^{1}_{\fm} (\Tor_1^R(R/ \omega_R, \omega_R))=0$.
We put this into the exact sequence induced by a)  to observe
that $\HH^{1}_{\fm} ( \omega_R\otimes  \omega_R)\cong\HH^{1}_{\fm} (L).$
By b) we have 
$0=\HH^0_{\fm}( \omega_R)\to \HH^0_{\fm}( \omega_R/ \omega_R^2)\to \HH^1_{\fm}(L)\to \HH^1_{\fm}( \omega_R)=0$.
It follows that $\HH^1_{\fm}( \omega_R\otimes_R   \omega_R)\cong \HH^0_{\fm}( \omega_R/ \omega_R^2)$ which is of finite length. The particular case follows from  the generalized Cohen--Macaulay property, see \cite{sch}.
\end{proof}

The $n$-th symbolic power of $I$ is denoted by $I^{(n)}$.

 \begin{corollary}\label{sym}
Adopt the notation of Corollary \ref{o2}.  If $d>1$ and $ \omega_R\otimes  \omega_R$ is quasi-Buchsbaum, then $\fm \omega_R^{(2)}\subset \omega_R^2$.
 \end{corollary}

\begin{proof}
	Without loss of generality we may and do assume that $R$ is not Gorenstein.  In  Corollary \ref{o2} we proved that  $\HH^1_{\fm}( \omega_R\otimes_R   \omega_R)\cong \HH^0_{\fm}( \omega_R/ \omega_R^2)$.
 Since $R$ is not Gorenstein, we  deduce that  $R/ \omega_R$ is Gorenstein and of dimension $d-1>0$.
	Thus, $\HH^0_{\fm}( R/ \omega_R )=0$. We look at the exact sequence  $0\to \omega_R/ \omega_R^2 \to R/ \omega_R^2\to R/ \omega_R\to 0$ and the induced long exact sequence of
	local cohomology modules. Then we have $0\to\HH^0_{\fm}(  \omega_R/ \omega_R^2) \to \HH^0_{\fm}( R/ \omega_R^2)\to\HH^0_{\fm}( R/ \omega_R )=0 $.
Recall that $ \omega_R$ is locally complete-intersection (in fact locally principal). This implies that $\HH^0_{\fm}( R/ \omega_R^2)\cong  \omega_R^{(2)}/ \omega_R^2 $, e.g., see  \cite{A}. In sum,
	$$\HH^1_{\fm}( \omega_R\otimes_R   \omega_R)\cong\HH^0_{\fm}(\omega_R/ \omega_R^2)\cong\HH^0_{\fm}( R/ \omega_R^2)\cong\omega_R^{(2)}/ \omega_R^2.$$ 
As   $\fm\HH^1_{\fm}( \omega_R\otimes_R   \omega_R)=0$,  we conclude that $\fm \omega_R^{(2)}\subset \omega_R^2$.
\end{proof}

\section{The h-reduced problem}
In this section we deal with Conjecture \ref{1.7}. 
Recall that $Q(R)$ means the fraction field of $R$.
\begin{definition} (Matlis) Suppose $R$ is a domain.
An $R$-module $M$ is called $h$-reduced if $\Hom_R(Q,M)=0$.
\end{definition}
Recall that $(R,\fm)$ is called analytically unramified in co-dimension one, if for any $\fp$ of height one, $R_{\fp}$ is analytically unramified. 
\begin{proposition}\label{sol}
	Suppose $(R,\fm)$ is analytically unramified in co-dimension one. If $R$ is $(S_2)$
	then $\frac{\overline{R}}{R}$ is $h$-reduced.
\end{proposition}

\begin{proof}
	Without loss of generality we may assume that $\frac{\overline{R}}{R}$
	is not zero. The proof is proceeds by induction on $d:=\dim(R)$.
	Suppose first that $d\leq1$. In this case, and due to our assumption,
	$\overline{R}$ is finitely generated as an $R$-module. This in turn
	implies that there is some $0\neq c$ so that $c\frac{\overline{R}}{R}=0$.
	Let $f:Q\lo \frac{\overline{R}}{R}$ be given. Then $f(x)=cf(xc)\subseteq c \frac{\overline{R}}{R}=0$ and so $f=0$. Now, assume $d>1$ and the desired claim has been hold for rings of dimension less than $d$. 
Suppose on the way of contradiction that
there is a nonzero $f:Q\lo \frac{\overline{R}}{R}$. Set $I:=\im({f})\subseteq\frac{\overline{R}}{R}$, and recall that it is nonzero, as $f\neq 0$.
Suppose $\Supp(I)\subseteq\{\fm\}$. Clearly, this gives $\Supp(I)=\{\fm\}$. 
Then for any nonzero and finitely generated submodule $J\subseteq I$ we have $\depth(J)=0$, and so $$\frac{R}{ \fm} \hookrightarrow J\subseteq I\subseteq\frac{\overline{R}}{R}.$$
Consequently, $H^0_{\fm}(\frac{\overline{R}}{R})
\neq 0$. Now, we look at 
 the exact sequence $0\lo R\lo \overline{R}\lo \frac{\overline{R}}{R}\lo 0$.
This gives the exact sequence $$0=H^0_{\fm}(\overline{R} )\lo H^0_{\fm}(\frac{\overline{R}}{R})\lo H^1_{\fm}(R)=0,$$
since both of $\overline{R}$ and $R$ are of depth strictly bigger than one. From this contradiction, 
we can assume in addition that $\Supp(I)\nsubseteq\{\fm\}$.
 From this, we can find  $\fp\in\Supp(I)\setminus\{\fm\}$. We denote the natural map $Q(R)\to \frac{\overline{R}}{R}$ by $\bar{f}$ and recall that it is nonzero, as $f\neq 0$.
	Now, we look at:
	$$\xymatrix{
		&Q(R)\ar[r]^{\bar{f}} \ar[rrrd]_{\exists\neq 0}   &I \ar[r]^{\neq0} &I_{\fp} \ar[r]^{\subseteq}   &  (\frac{\overline{R}}{R})_{\fp}\ar[d]^{\cong}\\&&&&
		\frac{\overline{R_{\fp}}}{R_{\fp}}
	 &&&}$$
Since	$Q(R)=Q(R_{\fp})$ we get a nonzero  map 	$Q(R_{\fp})\to \frac{\overline{R_{\fp}}}{R_{\fp}}$. This contradicts the inductive hypothesis.
So, $f=0$ and the desired claim holds.
\end{proof}

 \end{document}